\documentclass[12 pt]{amsart}
\usepackage{amssymb}
\usepackage{amsmath}
\usepackage{amscd}
\usepackage{graphicx}
\usepackage{latexsym}

\calclayout\evensidemargin .6in

\def\C{\mathbb{C}}
\def\R{\mathbb{R}}

\def\D{\mathbb{D}}

\newcommand{\ZZ}{\mathbb{Z}}

\newtheorem{theorem}{Theorem}[section]
\newtheorem{lemma}[theorem]{Lemma}

\newtheorem{corollary}[theorem]{Corollary}

\theoremstyle{definition}

\newtheorem{remark}[theorem]{Remark}
\newtheorem{definition}[theorem]{Definition}

\title[EXISTENCE OF COMPATIBLE CONTACT STRUCTURES ON $G_2-$MANIFOLDS]
{EXISTENCE OF COMPATIBLE CONTACT STRUCTURES ON $G_2-$MANIFOLDS}

\author{M. Firat Arikan}
\address{Department of Mathematics, University of Rochester, Rochester NY, USA}
\email{arikan@math.rochester.edu}
\thanks{The first named author is partially supported by NSF FRG grant DMS-1065910}

\author{Hyunjoo Cho}
\address{Department of Mathematics, University of Rochester, Rochester NY, USA}
\email{cho@math.rochester.edu}

\author{Sema Salur}
\address{Department of Mathematics, University of Rochester, Rochester NY, USA}
\email{salur@math.rochester.edu}
\thanks{The third named author is partially supported by NSF grant DMS-1105663}


\begin{document}
\begin{abstract}
In this paper, we show the existence of (co-oriented) contact structures
on certain classes of $G_2$-manifolds, and that these two structures
are compatible in certain ways. Moreover, we prove that any
seven-manifold with a spin structure (and so any manifold with $G_2$-structure)
admits an almost contact structure. We also construct explicit almost contact metric
structures on manifolds with $G_2$-structures.
\end{abstract}

\maketitle

%
%

\section{Introduction}

Let $(M,g)$ be a Riemannian $7$-manifold whose holonomy group $Hol(g)$ is the exceptional Lie group $G_2$ (or, more generally, a subgroup of $G_2$). Then $M$ is naturally equipped with a covariantly  constant $3$-form $\varphi$ and $4$-form $*\varphi$. We can define $(M,\varphi,g)$ as the $G_2$-manifold with $G_2$ structure $\varphi$.

We can also define a (co-oriented) contact manifold as a pair $(N,\xi)$ where $N$ is an odd dimensional manifold and $\xi$, called a (co-oriented) contact structure, is a totally non-integrable (co-oriented) hyperplane distribution on $N$.

In dimension 7, so far contact geometry and $G_2$ geometry have
been studied independently and each geometry has very distinguished characteristics which are rather different than those in the other. A basic example of such differences is the following: In contact geometry there are no local invariants, in other words, every contact $7$-manifold is locally contactomorphic to $\R^7$ equipped with the standard contact structure. On the other hand, in $G_2$ geometry it is the $G_2$ structure itself that determines how local neighborhoods of points look like, and as a result, manifolds with $G_2$ structures can look the same only at a point, \cite{Geiges}, \cite{Karigiannis}.

The aim of this paper is to initiate a new interdisciplinary research area between contact and $G_2$ geometries. More precisely, we study the existence of (almost) contact structures on 7-dimensional manifolds with (torsion free) $G_{2}$-structures.

The paper is organized as follows: After the preliminaries (Section \ref{sec:Preliminaries}), we show the existence of almost contact structures on $7$-manifolds with spin structures in Section \ref{sec:Almost_contact}. In particular, we prove the following theorem:

\vspace{.1in}

\noindent {\bf Theorem:} {\em Every manifold with $G_2$-structure admits an almost contact structure.}

\vspace{.1in}

In Section \ref{sec:Compatibility_Motivating example}, we define $A$- and $B$-compatibility between contact and $G_{2}$ structures, and also present the motivating example for $\mathbb{R}^7$. We also prove the nonexistence result:

\vspace{.1in}

\noindent {\bf Theorem:} {\em Let $(M,\varphi)$ be a manifold with $G_2$-structure such that $d\varphi=0$. If $M$ is closed (i.e., compact and $\partial M= \emptyset$), then there is no contact structure on $M$ which is A-compatible with $\varphi$.}

\vspace{.1in}

In Section \ref{sec:Explicit_Almost_Contact}, for any non-vanishing vector field $R$ on a manifold $M$ with $G_2$-structure $\varphi$, we construct explicit almost contact structure, denoted by $(J_R,R,\alpha_{R},g_{\varphi})$, and indeed prove the following theorems:

\vspace{.1in}

\noindent {\bf Theorem:} {\em  Let $(M,\varphi)$ be a manifold with $G_2$-structure. Then the quadruple $(J_R,R,\alpha_{R},g_{\varphi})$ defines an almost contact metric structure on $M$
for any non-vanishing vector field $R$ on $M$. Moreover, such a structure exists on any manifold with $G_2$-structure.}

\vspace{.1in}

\noindent {\bf Theorem:} {\em Let $(M,\varphi)$ be a manifold with $G_{2}$-structure.
Suppose that $\xi$ is a contact structure on $M$ such that $(J_{R},R, \alpha_{R},
g_{\varphi})$ is an associated almost contact metric structure for
$\xi$. Then $\xi$ is $A$-compatible.}

\vspace{.1in}

In Section \ref{sec:Contact$-G_2-$structures}, we define contact$-G_2-$structures on $7$-manifolds and analyze their relations with $A$-compatible contact structures, the main results of that section are:
\vspace{.1in}

\noindent {\bf Theorem:} {\em Let $(M,\varphi)$ be a manifold with $G_2$-structure. Assume that there are nowhere-zero vector fields $X$, $Y$ and $Z$ on $M$ satisfying
$\iota _Z \varphi=Y^{\flat} \wedge X^{\flat}$
where $X^{\flat}$ (resp. $Y^{\flat}$) is the covariant 1-form of $X$ (resp. $Y$) with respect to the $G_2$-metric $g_{\varphi}$. Also suppose that
$d (i_{X}i_{Y}\varphi)=i_{X}i_{Y}\ast \varphi.$
Then the 1-form $\alpha := Z^{\,\flat}=g_{\varphi}(Z,\cdot)$ is a contact form on $M$ and it  defines an $A$-compatible contact structure $\emph{Ker}(\alpha)$ on $(M,\varphi)$.}

\vspace{.1in}

\noindent {\bf Theorem:} {\em Let $(\varphi, R, \alpha,f,g)$ be a contact$-G_2-$structure on a smooth manifold $M^7$. Then $\alpha$ is a contact form on $M$. Moreover, $\xi=$\emph{Ker}$(\alpha)$ is an A-compatible contact structure on $(M,\varphi)$. In particular, if $M$ is closed, then it does not admit a contact$-G_2-$structure with $d\varphi=0$.}

\vspace{.1in}

\noindent {\bf Theorem:} {\em Let $(M,\varphi)$ be any manifold with $G_2$-structure. Then every A-compatible contact structure on $(M,\varphi)$ determines a contact$-G_2-$structure on $M$.}

\vspace{.1in}

Finally, in Section \ref{sec:Examples}, we present some examples of $A$-compatible structures and contact$-G_2-$structures.

%
%

\section{Preliminaries} \label{sec:Preliminaries}

\subsection{$G_2$-structures and $G_2$-manifolds}
A smooth $7$-dimensional manifold $M$ has a {\it $G_2$-structure},
if the structure group of $TM$ can be reduced to $G_2$. The group
$G_2$ is one of the five exceptional Lie groups which is the group
of all linear automorphisms of the imaginary octonions $im
\mathbb{O}\cong \mathbb{R}^7$ preserving a certain cross product.
Equivalently, it can be defined as the subgroup of
$GL(7,\mathbb{R})$ which preserves the $3$-form
$$\varphi_0=e^{123}+e^{145}+e^{167}+e^{246}-e^{257}-e^{347}-e^{356}$$
where $(x_1,...,x_7)$ are the coordinates on $\mathbb{R}^7$, and
$e^{ijk}=dx^i \wedge dx^j \wedge dx^k$. As an equivalent definition,
a manifold with a $G_2$-structure $\varphi$ is a pair $(M,\varphi)$,
where $\varphi$ is a $3$-form on $M$, such that $(T_pM,\varphi)$ is
isomorphic to $(\mathbb{R}^7,\varphi_0)$ at every point $p$ in $M$.
Such a $\varphi$ defines a Riemannian metric $g_{\varphi}$ on $M$. We say
$\varphi$ is \emph{torsion-free} if $\nabla \varphi=0$ where
$\nabla$ is the Levi-Civita connection of $g_{\varphi}$. A Riemannian manifold
with a torsion free $G_2$-structure is called a {\it
$G_2$-manifold}. Equivalently, the pair $(M,\varphi)$ is called a
$G_2$-manifold if its holonomy group (with respect to $g_{\varphi}$) is a
subgroup of $G_2$. As an another characterization, one can show that
$\varphi$ is torsion-free if and only if $d\varphi=d(\ast\varphi)=0$
where $``\ast"$ is the Hodge star operator defined by the metric
$g_{\varphi}$.\\

The $3$-form $\varphi$ also determines the cross product and the
orientation top (volume) form Vol on $M$. In fact, for any vector
fields $u,v,w$ on $M$, we have
\begin{equation} \label{eqn:Cross_Product}
\varphi(u,v,w)=g_{\varphi}(u \times v,w),
\end{equation}
\begin{equation} \label{eqn:Metric_Volume}
(\iota_u \varphi) \wedge (\iota_v \varphi) \wedge \varphi=6g_{\varphi}(u,v)\;
\textrm{Vol}.
\end{equation}
Also we will make use of the following formula as well:
\begin{equation} \label{eqn:Double_Cross_Product}
u \times (u \times v)=-\|u\|^2v+g_{\varphi}(u,v)u.
\end{equation}

See \cite{Bryant}, \cite{Bryant2}, \cite{Karigiannis}  and \cite{Joyce} for more details
on $G_2$ geometry.\\

\subsection{Contact and almost contact structures}
A \emph{contact structure} on a smooth
$(2n+1)$-dimensional manifold $M$ is a global $2n$-plane field
distribution $\xi$ which is totally non-integrable.
Non-integrability condition is equivalent to the fact that locally
$\xi$ can be given as the kernel of a 1-form $\alpha$ such that
$\alpha \wedge (d\alpha)^n \neq0$. If $\alpha$ is globally defined
(in such a case, it is called a \emph{contact form}), then one can
define the \emph{Reeb vector field} of $\alpha$ to be the unique global nowhere-zero
vector field $R$ on $M$ satisfying the equations
\begin{equation} \label{eqn:Defining_Reeb}
\iota_Rd\alpha=0, \quad \alpha(R)=1
\end{equation}
where $`` \iota "$ denotes the interior product.\\

Using $R$, we can co-orient $\xi$ and, as a result, the structure
group of the tangent frame bundle can be reduced to $U(n) \times 1$.
Such a reduction of the structure group is called an \emph{almost
contact structure} on $M$. Therefore, for the existence of a
co-oriented contact structure on $M$, one should first ask the
existence of an almost contact structure. We refer the reader to
\cite{Blair} and \cite{Geiges} for more on contact geometry.

\begin{definition} [\cite{GrayJW}] \label{def:Almost_contact_1}
Let $M^{2n+1}$ be a smooth manifold. If the structure group of its
tangent bundles $TM^{2n+1}$ reduces to $U(n)\times 1$, then
$M^{2n+1}$ is said to have  \emph{an almost contact structure}.
\end{definition}

%
%

\section{Almost contact structures on $7$-manifolds with a spin structure} \label{sec:Almost_contact}

Although no explicit description is given, nevertheless the
following result shows the existence of almost contact structures
not only on manifolds with $G_2$-structures but also on a much wider
family of 7-manifolds. Recall that if a manifold admits a spin
structure, then its second Stiefel-Whitney class is zero.

\begin{theorem} \label{thm:Spin_almost_contact}
Every $7$-manifold with a spin structure admits an almost contact
structure.
\end{theorem}
\begin{proof}
Assume that $M$ is a $7$-manifold with spin structure. By
definition, $M$ admits an almost contact structure if and only if
the structure group of $TM$ can be reduced to  $U(3)\times 1$.
Equivalently, the associated fiber bundle $TM[SO(7) / U(3)]$ with
fiber $ SO(7)/U(3)$ admits a cross-section \cite{Steenrod}. If
$s$ is a cross section of fiber bundle over the the $(i-1)$-skeleton
of $M$, then the cohomology class
\begin{equation*}
 \begin{CD}
 o^{i}(s)\in H^{i}(M,\pi_{i-1}(SO(7)/U(3)).
 \end{CD}
\end{equation*}
is the obstruction to extending $s$ over the $i$-skeleton. Since we
have $$\pi_i(SO(7) / U(3))= 0$$ unless $i=2,6$, the only
obstructions to the existence of such a cross section arise in
$H^i(M, \ZZ)$ for $i=3,7$. In \cite{Massey}, Massey shows that these
obstructions are the integral Stiefel-Whitney classes of the
associated dimensions. Recall that the integral Stiefel-Whitney
classes are defined as the images $\beta(w_i)$ of the
Stiefel-Whitney classes under the Bockstein homomorphism. Here the
Bockstein homomorphism is the connecting homomorphism
$\beta:H^{i}(M, \ZZ / 2 \ZZ ) \rightarrow H^{i+1}(M, \ZZ ) $ which
arises from the short exact sequence
\begin{equation*}
\begin{CD}
0 @>>> \ZZ @> \times 2 >> \ZZ @>>> \ZZ / 2 \ZZ @>>>0.
\end{CD}
\end{equation*}
Therefore, the obstructions $o^3, o^7$ to the existence of an almost
contact structures on $7$-manifolds are 2-torsion classes.

Now we know that $w_2(M)=0$ (since $M$ is spin), and hence the third
integral Stiefel-Whitney class vanishes, i.e., $o^3=W_3(M)=
\beta(w_2)=0$. Therefore, the only obstruction lies in the
cohomology group $H^7(M)$.

We consider the following cases: First, if $M$ is a closed manifold,
then by Poincar{\'e} duality $H^7(M) \cong H_0(M) \cong \ZZ$ and
hence the top dimensional obstruction $o^7$ vanishes. Secondly, if
$M$ has a boundary, then (again by the duality) we have $o^7 \in
H^7(M) \cong H_0(M,\partial M) \cong 0$. Now, if $M$ is non-compact
without a boundary, then the cohomology group $H^7(M)\cong
(H_{cs}^0(M))^{\ast}$ where $H_{cs}$ denotes the compactly supported
cohomology. Hence, it is torsion-free.

\end{proof}

Since every manifold with $G_2$-structure is spin, we get

\begin{corollary} \label{cor:G_2_almost_contact}
Every manifold with $G_2$-structure admits an almost contact
structure.\qed
\end{corollary}

%
%

\section{Compatibility and the motivating example} \label{sec:Compatibility_Motivating example}

Assuming the existence of a contact structure on a manifold with a
$G_2$-structure, we can also ask if and how these two
structures are related. We define \textit{two} different notions of
\emph{compatibility} between them as follows:

\begin{definition} \label{def:A-Compatibility}
A (co-oriented) contact structure $\xi$ on $(M,\varphi)$ is said to
be \emph{A-compatible} with the $G_2$-structure $\varphi$ if there exist a vector field $R$ on $M$ and a nonzero function $f:M\rightarrow \mathbb{R}$ such that $d\alpha=\iota_R \varphi$ for some contact form $\alpha$ for $\xi$ and $fR$ is the Reeb vector field of a contact form for $\xi$.
\end{definition}

\begin{definition} \label{def:B-Compatibility}
A (co-oriented) contact structure $\xi$ on $(M,\varphi)$ is said to
be \emph{B-compatible} with the $G_2$-structure $\varphi$ if there
are (global) vector fields $X$, $Y$ on $M$ such that $\alpha=\iota_Y
\iota_X \varphi$ is a contact form for $\xi$.
\end{definition}

\vspace{.1in}

In this paper, we will mainly consider $A$-compatible contact structures. We remark that if $\varphi$ is torsion-free or at least $d\varphi=0$, then Definition \ref{def:A-Compatibility} makes sense only if $M$ is noncompact or compact with boundary. Indeed, we can easily prove the following:

\vspace{.1in}

\begin{theorem} \label{thm:A-Comp_can_not_Closed}
Let $(M,\varphi)$ be a manifold with $G_2$-structure such that $d\varphi=0$. If $M$ is closed (i.e., compact and $\partial M= \emptyset$), then there is no contact structure on $M$ which is A-compatible with $\varphi$.
\end{theorem}
\begin{proof} Suppose $\xi$ is an A-compatible contact structure on
$(M,\varphi)$. Therefore, $d\alpha=\iota_R \varphi$ for some contact
form $\alpha$ for $\xi$ and some nonvanishing vector field $R$.
Using the equation (\ref{eqn:Metric_Volume}), we have
$$d\alpha \wedge d\alpha \wedge \varphi=(\iota_R \varphi) \wedge (\iota_R
\varphi) \wedge \varphi=6\|R\|^2\textrm{ Vol}.$$ Since $d\varphi=0$,
we have $d\alpha \wedge d\alpha \wedge \varphi=d(\alpha \wedge
d\alpha \wedge \varphi)$. Now by Stoke's Theorem,
$$0 \lneqq \int_M 6\|R\|^2\textrm{ Vol}=\int_M d(\alpha \wedge d\alpha \wedge \varphi)=\int_{\partial M}\alpha \wedge d\alpha \wedge \varphi=0$$
(as $\partial M=\emptyset$). This gives a contradiction.

\end{proof}

For another application of this argument on specific vector fields on manifolds with $G_2$ structures, see \cite{CST1}.\\

We now explore the relation between the standard contact structure
$\xi_0$ and the standard $G_2$-structure $\varphi_0$ on $\R^7$.
Indeed, the notion of A- and B-compatibility
relies on this motivating example.\\

Fix the coordinates $(x_1, x_2, x_3, x_4, x_5, x_6, x_7)$ on
$\mathbb{R}^7$. In these coordinates,
$$\varphi_0=e^{123}+e^{145}+e^{167}+e^{246}-e^{257}-e^{347}-e^{356}$$
where $e^{ijk}$ denotes the 3-form $dx_i \wedge dx_j \wedge dx_k$.
Consider the standard contact structure $\xi_0$ on $\mathbb{R}^7$ as
the kernel of the 1-form
$$\alpha_0=dx_1-x_3dx_2-x_5dx_4-x_7dx_6.$$

For simplicity, through out the paper we will denote $\partial / \partial x_i$ by $\partial
x_i$ (so we have $dx_i(\partial x_j)=\delta_{ij}$). Consider the
vector fields
\begin{center}
$R=\partial x_1$, $X=\partial x_7$ and $Y=-x_7\partial
x_1+x_5\partial x_3-x_3\partial x_5- \partial x_6 + f\partial x_7$
\end{center}
where $f:\mathbb{R}^7  \rightarrow \mathbb{R}$ is any smooth
function (in fact, it is enough to take $f\equiv 0$ for our
purpose). By a straightforward computation, we see that
$$d\alpha_0=\iota_R (\varphi_0), \quad \alpha_0=\iota_Y\iota_X (\varphi_0).$$
Also observe that $R$ is the Reeb vector field of $\alpha_0$. Note that this contact structure is not unique A-compatible with $\varphi_{0}$.  In fact we have  various
ways of choosing the contact structures by rotating indexes and signes. For example, the contact structure $\alpha=dx_2+x_3dx_1-x_6dx_4+x_7dx_5$ with $R=\partial x_2$ is another A-compatible contact structure with $\varphi_0$ and by choosing  two vectors $X=\partial x_7, Y=\partial x_5-x_3\partial x_6+x_6\partial x_3-x_7\partial x_2+f \partial x_7$  it is easily seen as being B-compatible with $\varphi_{0}$.
Therefore, we have proved:

\begin{theorem}
There are contact structures $\xi$ on $\mathbb{R}^7$ which are both A-
and B-compatible with the standard $G_2$-structure $\varphi_0$.\qed
\end{theorem}

%
%

\section{An explicit almost contact metric structure} \label{sec:Explicit_Almost_Contact}

We first give an alternative definition of an almost contact
structure, and then construct an explicit almost contact structure
on a manifold with $G_2$-structure. The reader is referred to
\cite{Blair} for the equivalence between the previous definition (Definition \ref{def:Almost_contact_1}) and
this new one.

\begin{definition} [\cite{Sasaki}] \label{def:Almost_contact_2}
An \emph{almost contact structure} on a differentiable manifold
$M^{2n+1}$ is a triple $(J,R,\alpha)$ consists of a field $J$ of
endomorphisms of the tangent spaces, a vector field $R$, and a
$1$-form $\alpha$ satisfying
\begin{itemize}
\item[(i)] $\alpha(R)=1$,
\item[(ii)] $J^2=-I+\alpha \otimes R$
\end{itemize}
where $I$ denotes the identity transformation.
\end{definition}

For completeness, we provide the proof of the following lemma.

\begin{lemma} [\cite{Sasaki}] \label{cor:G_2_almost_contact}
Suppose that $(J,R,\alpha)$ is an almost contact structure on
$M^{2n+1}$. Then $J(R)=0$ and $\alpha\circ J=0$
\end{lemma}

\begin{proof}
Since $J^{2}(R)=-R+\alpha(R)R=-R+1\cdot R=0$, we have either $J(R)=0$ or $J(R)$
is nonzero vector field whose image is $0$. Suppose $J(R)$ is
nonzero vector field which is mapped to $0$ by $J$. Then from
$$0=J^{2}(J(R))=-J(R)+\alpha (J(R))\cdot R$$
we get $J(R)=\alpha (J(R))\cdot R$, and so $\alpha (J(R))\neq0$ (as $J(R)\neq0$). But then
$$J^{2}(R)=J(J(R))=J(\alpha(J(R))R)=\alpha(J(R))\cdot J(R)=
[\alpha(J(R))]^{2}\cdot R \neq 0$$
which contradicts to assumption that $J^2(R)=J(J(R))=0$. Hence, we conclude that $J(R)=0$ must be the case.\\

Now for any vector $X$, we see that
\begin{equation*}
 \begin{CD}
 J^{3}(X)=J(J^{2}(X))=J((-X)+\alpha(X)R)=-J(X)+J(\alpha(X)R)
 \end{CD}
\end{equation*}
 and also we have
\begin{equation*}
 \begin{CD}
 J^{3}(X)=J^{2}(J(X))=-J(X)+\alpha (J(X)) R.
 \end{CD}
\end{equation*}
So combining these we compute
\begin{eqnarray}
\nonumber \alpha (J(X)) R &=&J^{3}(X)+J(X)\\
\nonumber &=&-J(X)+J(\alpha(X)R)+J(X)=J(\alpha(X)R).
\end{eqnarray}
But using the fact $J(R)=0$ we have $$J(\alpha (X)R)=\alpha(X)J(R)=0.$$ Therefore, $\alpha (J(X))=0$ as $R\neq0$. Hence, $\alpha \circ J=0$ for any vector $X$.

\end{proof}

We can also introduce a Riemannian metric into the picture as
suggested in the following definition.

\begin{definition} [\cite{Sasaki}] \label{def:Almost_contact_metric}
An \emph{almost contact metric structure } on a differentiable
manifold $M^{2n+1}$ is a quadruple $(J,R,\alpha,g)$ where
$(J,R,\alpha)$ is an almost contact structure on $M$ and $g$ is a
Riemannian metric on $M$ satisfying
\begin{equation} \label{eqn:Compatible_Metric}
g(Ju,Jv)=g(u,v)-\alpha(u)\alpha(v)
\end{equation}
for all vector fields  $u,v$ in $TM$. Such a $g$ is called a
\emph{compatible metric}.
\end{definition}

\begin{remark}
Every manifold with an almost contact structure admits a compatible
metric (see \cite{Blair}, for a proof). Also setting $u=R$ in
Equation (\ref{eqn:Compatible_Metric}) gives
$g(JR,Jv)=g(R,v)-\alpha(R)\alpha(v)$. Since $J(R)=0$, an immediate
consequence is that $\alpha$ is the covariant form of $R$, that
is, $\alpha(v)=g(R,v)$.
\end{remark}

\begin{definition} [\cite{Sasaki}] Let $M$ be an odd dimensional manifold, and $\alpha$ be a contact form on $M$ with the Reeb vector field $R$. Therefore, $d\alpha$ is
a symplectic form on the contact structure (or distribution) $\xi=
\textrm{Ker}(\alpha)$. We say that the triple $(J,R,\alpha)$ is an
\emph{associated almost contact structure} for $\xi$ if $J$ is
$d\alpha$-compatible almost complex structure on the complex bundle
$\xi$, that is
\begin{center}
$d\alpha(JX,JY)=d\alpha(X,Y)$ and $d\alpha(X,JX)>0$ for all $X,Y \in
\xi$.
\end{center}
Furthermore, if $g$ is a metric on $M$, we consider two equations :
\begin{equation}
g(JX,JY)=g(X,Y)-\alpha(X)\alpha(Y)
\end{equation}
\begin{equation}
d\alpha(X,Y)=g(JX,Y)
\end{equation}
for all $X,Y \in TM$. We say that $(J,R,\alpha, g)$ is an
\emph{associated almost contact metric structure} if two equations
(6) and (7) hold. In this case, $g$ is called an \emph{associated
metric}.
\end{definition}

Suppose that $(M,\varphi)$ is a manifold with $G_2$-structure. There
might be many ways to construct almost contact metric structures on
$(M,\varphi)$. Here we give a particular way of constructing almost
contact metric structures on $(M,\varphi)$. Denote the Riemannian
metric and the cross product (determined by $\varphi$) by
$g_{\varphi}=\langle\cdot,\cdot\rangle_{\varphi}$ and
$\times_{\varphi}$, respectively. Suppose that $R$ is a nowhere vanishing vector field on $M$. By normalizing $R$ using $g_{\varphi}$, we may assume that $\|R\|=1$. Then using the metric, we define the
$1$-form $\alpha_{R}$ as the metric dual of $R$, that is,
$$\alpha_{R}(u)=g_{\varphi}(R,u)=\langle R,u\rangle_{\varphi}.$$
Moreover, using the cross product and $R$, we can define an
endomorphism $J_R:TM\rightarrow TM$ of the tangent spaces by
$$J_R(u)=R \times_{\varphi} u.$$ Note that $J_R(R)=0$, and so $J_R$,
indeed, defines a complex structure on the orthogonal complement
$R^{\perp}$ of $R$ with respect to $g_{\varphi}$. With these, we have

\begin{theorem} Let $(M,\varphi)$ be a manifold with $G_2$-structure. Then the quadruple $(J_R,R,\alpha_{R},g_{\varphi})$ defines an almost contact metric structure on $M$
for any non-vanishing vector field $R$ on $M$. Moreover, such a structure exists on any manifold with $G_2$-structure.
\end{theorem}

\begin{proof}
As before, we will assume that $R$ is already normalized using $g_{\varphi}$. First, note that $\alpha_{R}(R)=g_{\varphi}(R,R)=\|R\|^2=1$. Also we have
$$J^2_{R}(u)=J_{R}(R\times_{\varphi}u)=R\times_{\varphi}(R\times_{\varphi}u)
=-\|R\|^{2}u+g_{\varphi}(R,u)R=-u+\alpha(u)R$$
where we made use of the identity (\ref{eqn:Double_Cross_Product}). This shows that the endomorphism $J_R:TM \to TM$ satisfies the condition $$J^2_{R}=-I+\alpha \otimes R.$$ Therefore, the triple $(J_R,R,\alpha_R)$ is an almost contact structure on $M$. Next, we check $g_{\varphi}$ is a compatible metric with this
structure. Using (\ref{eqn:Cross_Product}) and (\ref{eqn:Double_Cross_Product}), we compute
\begin{eqnarray}
\nonumber g_{\varphi}(J_Ru,J_Rv)&=&g_{\varphi}(R \times_{\varphi}
u,R \times_{\varphi} v)=\varphi(R,u,R \times_{\varphi}  v)
=-\varphi(R,R
\times_{\varphi} v,u)\\
\nonumber &=&-g_{\varphi}(R \times_{\varphi} (R \times_{\varphi}
v),u)=-g_{\varphi}(-\|R\|^2v+g_{\varphi}(R,v)R,u)\\
\nonumber &=&-g_{\varphi}(-v+g_{\varphi}(R,v)R,u)=g_{\varphi}(v,u)-g_{\varphi}(\alpha_{R}(v)R,u)\\
\nonumber &=&=g_{\varphi}(u,v)-\alpha_{R}(v)\underbrace{g_{\varphi}(R,u)}_{\alpha_R(u)}=g_{\varphi}(u,v)-\alpha_{R}(u)\alpha_{R}(v)
\end{eqnarray}
which holds for all vector fields  $u,v$ in $TM$. This proves that $g_{\varphi}$ satisfies (\ref{eqn:Compatible_Metric}). Hence, $(J_R,R,\alpha_{R},g_{\varphi})$ is an almost contact metric structure on $M$.

For the last statement, we know by \cite{Thomas} that there exists a
nowhere vanishing vector field $R$ on any $7$-dimensional manifold. In particular, $(J_R,R,\alpha_{R},g_{\varphi})$ can be constructed on any manifold $M$ with $G_2$-structure $\varphi$.

\end{proof}

\begin{theorem} \label{thm:almost_contact_A_compatible}
Let $(M,\varphi)$ be a manifold with $G_{2}$-structure, and
$(J_{R},R, \alpha_{R}, g_{\varphi})$ be an almost contact metric
structure on $M$ constructed as above. Suppose that $\xi$ is a
contact structure on $M$ such that $(J_{R},R, \alpha_{R},
g_{\varphi})$ is an associated almost contact metric structure for
$\xi$. Then $\xi$ is $A$-compatible.
\end{theorem}

\begin{proof}
By assumption $(J_{R},R, \alpha_{R}, g_{\varphi})$ is an associated
almost contact metric structure for $\xi$. Therefore, $g_{\varphi}$
is an associated metric and satisfies
\begin{center} $d\alpha_{R}(u,v)=g_{\varphi}(J_{R}(u),v)$ for all
$u,v \in TM.$
\end{center}
But then using the equation defining $J_{R}$ and (\ref{eqn:Cross_Product}), we obtain
\begin{center}$d\alpha_{R}(u,v)=g_{\varphi}(R
\times_{\varphi}u,v)=\varphi(R,u,v)=i_{R}\varphi(u,v),  \quad  \forall
u,v \in TM$.
\end{center}
Therefore, we have $d\alpha_{R}=i_{R}\varphi$. Also $R$ is the Reeb vector field of $\alpha_{R}$ by assumption. Hence, $\xi$ is $A$-compatible by definition.

\end{proof}

\begin{corollary}
Let $(M,\varphi)$ be a manifold with $G_{2}$-structure such that $d\varphi=0$,
and $(J_{R},R, \alpha_{R}, g_{\varphi})$ be an almost contact metric
structure on $M$ constructed as above. If $M$ is closed, then there is no contact
structure on $M$ whose associated almost contact metric structure is
$(J_{R},R, \alpha_{R}, g_{\varphi})$.
\end{corollary}

\begin{proof}On the contrary, suppose that $\xi= \textrm{Ker}(\alpha_R)$ is a
contact structure on a closed manifold $M$ equipped with a
$G_{2}$-structure $\varphi$ and $d\varphi=0$, and also that $(J_{R},R,
\alpha_{R}, g_{\varphi})$ is an associated almost contact metric
structure. Then, by Theorem \ref{thm:almost_contact_A_compatible},
$\xi$ is A-compatible, but this contradicts to Theorem
\ref{thm:A-Comp_can_not_Closed}.

\end{proof}

%
%

\section{Contact$-G_2-$structures on $7$-manifolds} \label{sec:Contact$-G_2-$structures}

Suppose that $(M,\varphi)$ is a manifold with $G_2$-structure. Let us recall the decomposition of the space $\Lambda^{2}$ of $2$-forms on $M$ obtained from $G_2$-representation and some other useful formulas which we will use. A good source for these is \cite{Bryant} and also \cite{Karigiannis}. According to irreducible $G_{2}$-representation, $\Lambda^{2}=\Lambda^{2}_{7}\oplus \Lambda^{2}_{14}$ where
\begin{eqnarray} \label{eqn:Decomposition_from_G_2_rep}
\nonumber \Lambda^{2}_{7} &=& \{ i_{v}\varphi; v \in \Gamma(TM)\} \\
\nonumber &=&\{\beta \in \Lambda^{2}; \ast(\varphi \wedge \beta)=-2\beta \}\\
 &=&\{\beta \in \Lambda^{2} ; \ast(\ast \varphi \wedge (\ast (\ast \varphi \wedge \beta)))=3\beta \}\\
\nonumber & & \\
\nonumber \Lambda^{2}_{14} &=& \{\beta \in \Lambda^{2} ; \ast \varphi \wedge \beta=0 \}\\
\nonumber &=& \{ \beta \in \Lambda^{2} ; \ast (\varphi \wedge \beta)=\beta \}
\end{eqnarray}

Also on any Riemannian $n$-manifold, for any $k$-form $\alpha$ and a vector field $v$, the following equalities hold:
\begin{equation} \label{eqn_k_v_with_Hodge_star}
i_{v}\ast \alpha = (-1)^{k} \ast(v^{\flat} \wedge \alpha)  \;\;\;\;\; \textrm{and}
\end{equation}
\begin{equation}\label{eqn_k_v_no_Hodge_star}
i_{v}\alpha=(-1)^{nk+n}\ast(v^{\flat} \wedge \ast \alpha).
\end{equation}

As a last one we recall a very useful equality: For any $k$-form $\lambda$, and any
$(n+1-k)$-form $\mu$ and any vector field $v$ on a smooth manifold
of dimension $n$, we have
\begin{equation} \label{eqn:Key_for_Showing_Contact}
(\iota_v \lambda) \wedge \mu=(-1)^{k+1}\lambda \wedge (\iota_v \mu).
\end{equation}

Now we are ready to prove:

\begin{theorem} \label{thm:G2_representation}
Let $(M,\varphi)$ be a manifold with $G_2$-structure. Assume that there are nowhere-zero vector fields $X$, $Y$ and $Z$ on $M$ satisfying
\begin{equation} \label{eqn:condition_1_thm_G_2_rep}
\iota _Z \varphi=Y^{\flat} \wedge X^{\flat}
\end{equation}
where $X^{\flat}$ (resp. $Y^{\flat}$) is the covariant 1-form of $X$ (resp. $Y$) with respect to the $G_2$-metric $g_{\varphi}$. Also suppose that
\begin{equation} \label{eqn:condition_2_thm_G_2_rep}
d (i_{X}i_{Y}\varphi)=i_{X}i_{Y}\ast \varphi.
\end{equation}
Then the 1-form $\alpha := Z^{\,\flat}=g_{\varphi}(Z,\cdot)$ is a contact form on $M$ and it  defines an $A$-compatible contact structure $\emph{Ker}(\alpha)$ on $(M,\varphi)$.
\end{theorem}

\begin{proof}
From (\ref{eqn:Decomposition_from_G_2_rep}) we know that $\iota _Z \varphi$ is an element of $\Lambda^{2}_{7}$. Set $\iota _Z \varphi =\beta \in \Lambda^{2}_{7}$, and so we have $\iota _Z \varphi=\beta=Y^{\flat} \wedge X^{\flat}$ by (\ref{eqn:condition_1_thm_G_2_rep}).
Also applying (\ref{eqn_k_v_with_Hodge_star}) twice gives
\begin{equation}
\nonumber i_{X}i_{Y}\ast \varphi=-i_{X}(\ast(Y^{\flat} \wedge \varphi))=-\ast(X^{\flat}\wedge Y^{\flat} \wedge \varphi)=\ast(Y^{\flat}\wedge X^{\flat}\varphi)=\ast(\beta \wedge \varphi)
\end{equation}
from which we get
\begin{equation}
i_{X}i_{Y}\ast \varphi=-2\beta
\end{equation}
where we use the second line in (\ref{eqn:Decomposition_from_G_2_rep}). Moreover, by (\ref{eqn_k_v_no_Hodge_star}) followed by (\ref{eqn_k_v_with_Hodge_star}),
\begin{equation}
i_{X}i_{Y}\varphi=i_{X}(\ast (Y^{\flat} \wedge \ast \varphi))=-\ast(X^{\flat} \wedge Y^{\flat} \wedge \ast \varphi)=\ast(\beta \wedge \ast \varphi).
\end{equation}

\noindent Now putting (14) and (15) into (\ref{eqn:condition_2_thm_G_2_rep}) gives us
\begin{equation}
d\ast(\beta \wedge \ast \varphi)=-2\beta=-2\,\iota _Z \varphi.
\end{equation}
Recall the formula $(i_{v}\varphi) \wedge \ast \varphi=3\ast v^{\flat}$ which is true for any vector field $v$. By taking $v=Z$, we compute the left-hand side in (16) as $$d\ast(\beta \wedge \ast \varphi)=d\ast(3 \ast Z^{\flat})=3\,dZ^{\flat}=3\,d\alpha.$$
Combining these together we obtain
\begin{equation}
d\alpha=-\frac{2}{3} \,\iota _Z \varphi.
\end{equation}

\noindent Next, consider the identity (\ref{eqn:Key_for_Showing_Contact})
by taking $\lambda=\varphi$, $v=Z$ and $\mu=\alpha\wedge (d\alpha)^2$: Using (17), we compute the left-hand side as $$(\iota_Z \varphi) \wedge \alpha\wedge (d\alpha)^2=-\frac{3}{2}\alpha\wedge (d\alpha)^3,$$
and the right-hand side as
$$\varphi \wedge \iota_Z (\alpha\wedge (d\alpha)^2)=\alpha(Z) \, \varphi \wedge d\alpha \wedge d\alpha=\frac{4}{9}\|Z\|^2 \, \varphi \wedge (\iota_Z \varphi) \wedge (\iota_Z \varphi).$$
Therefore, by using the identity (\ref{eqn:Metric_Volume}) in the right-hand side, we obtain
$$\alpha\wedge (d\alpha)^3=-\frac{16}{9}\|Z\|^4 \, \textrm{ Vol}.$$
Hence, we conclude that $\alpha\wedge (d\alpha)^3$ is a volume form on $M$ (as being a nonzero function multiple of the volume form Vol on $M$ determined by the metric $g_{\varphi}$).
Equivalently, $\alpha$ is a contact form on $M$. Moreover, it follows from (17) that $(1/\|Z\|^2)Z$ is the Reeb vector field of $\alpha$, i.e., it satisfies (\ref{eqn:Defining_Reeb}). Hence, Ker$(\alpha)$ is an $A$-compatible contact structure on $(M,\varphi)$.

\end{proof}

With the inspiration we get from the proof of Theorem \ref{thm:G2_representation}, we define a new structure on 7-manifolds as follows:

\begin{definition} \label{def:Contact_G2_structure}
Let $M^7$ be a smooth manifold. A \emph{contact$-G_2-$structure} on $M$ is a quintuple  $(\varphi,R,\alpha,f,g)$ where $\varphi$ is a $G_2$-structure, $R$ is a nowhere-zero vector field, $\alpha$ is a 1-form on $M$, and $f,g:M \to \R$ are nowhere-zero smooth functions such that
\begin{itemize}
\item[(i)] $\alpha(R)=f$
\item[(ii)] $d(g\,\alpha)= \iota_R \varphi$.
\end{itemize}
\end{definition}

Observe that we have already seen an example of a contact$-G_2-$structure in the above proof (of course under the assumptions of Theorem \ref{thm:G2_representation}) with $R=Z, \alpha=Z^{\,\flat},f=\|Z\|^2, g\equiv -3/2$. The reason why we call the quintuple $(\varphi, R, \alpha,f,g)$
``contact$-G_2-$structure'' is given by the following theorem.

\begin{theorem} \label{thm:Contact_G2_Str-->A_compatible}
Let $(\varphi, R, \alpha,f,g)$ be a contact$-G_2-$structure on a smooth manifold $M^7$. Then $\alpha$ is a contact form on $M$. Moreover, $\xi=$\emph{Ker}$(\alpha)$ is an A-compatible contact structure on $(M,\varphi)$. In particular, if $M$ is closed, then it does not admit a contact$-G_2-$structure with $d\varphi=0$.
\end{theorem}

\begin{proof}
We first show that $\alpha$ is a contact form on $M$. Consider the $1$-form $$\alpha\,':=g\,\alpha.$$ Note that $\textrm{Ker}(\alpha)=\textrm{Ker}(\alpha\,')$ as $g$ is a nowhere-zero function. Therefore, if we show that $\alpha\,'$ is a contact form on $M$, then it will imply that so is $\alpha$. The conditions in Definition \ref{def:Contact_G2_structure} translate into $$\alpha\,'(R)=fg \quad \textrm{and} \quad d\alpha\,'= \iota_R \varphi.$$ Also from the equation (\ref{eqn:Metric_Volume}) we get
$$(d\alpha\,')^2 \wedge \varphi=(\iota_R \varphi) \wedge (\iota_R \varphi) \wedge
\varphi=6\|R\|^2\textrm{ Vol}.$$
Now if we write the equation (\ref{eqn:Key_for_Showing_Contact}) by
taking $\lambda=\varphi, \mu=\alpha\,' \wedge (d\alpha\,')^2$ and $v=R$,
then the left-hand side gives $$(\iota_R \varphi) \wedge \alpha\,'
\wedge (d\alpha\,')^2=(d\alpha\,') \wedge \alpha\,' \wedge (d\alpha\,')^2=\alpha\,' \wedge (d\alpha\,')^3,$$ and from the right-hand side we have
$$\varphi \wedge \iota_R (\alpha\,' \wedge (d\alpha\,')^2)=\alpha\,'(R) \,\varphi \wedge (d\alpha\,')^2=fg\,\varphi \wedge (d\alpha\,')^2=6\,fg\|R\|^2\textrm{ Vol}.$$
Therefore, we conclude $$\alpha\,' \wedge (d\alpha\,')^3=6\,fg\|R\|^2\textrm{ Vol}$$ which implies that $\alpha\,'$ (and so $\alpha$) is a contact form on $M$ as $6\,fg\|R\|^2$ is a nowhere-zero function on $M$.

Next, we consider the vector field $R\,'=(1/fg)R$. Clearly, $\alpha\,'(R\,')=1$. Also we compute
$$\iota_{R\,'} d\alpha\,'=(1/fg)\,\iota_R d\alpha\,'=(1/fg)\,\iota_R (\iota_R \varphi)=0$$ as $\varphi$ is skew-symmetric. Therefore, $R\,'$ is the Reeb vector field of $\alpha\,'$, and so $\xi=\textrm{Ker}(\alpha\,')$ is an A-compatible contact structure on $(M,\varphi)$ by definition. Finally, the last statement now follows from Theorem \ref{thm:A-Comp_can_not_Closed}.

\end{proof}

The next result shows that we can go also in the reverse direction.

\begin{theorem} \label{thm:A_compatible-->Contact_G2_Str}
Let $(M,\varphi)$ be any manifold with $G_2$-structure. Then every A-compatible contact structure on $(M,\varphi)$ determines a contact$-G_2-$structure on $M$.
\end{theorem}

\begin{proof} Let $\xi$ be a given A-compatible contact structure on $(M,\varphi)$. By definition, there exist a non-vanishing vector field $R$ on $M$, a contact form $\alpha$ for $\xi$ and a nowhere-zero function $h:M \to \R$ such that $d\alpha=\iota_R \varphi$ and $hR$ is the Reeb vector field of some contact form (possibly different than $\alpha$) for $\xi$. Being a Reeb vector field, $hR$ is transverse to the contact distribution $\xi$. Therefore, $R$ is also transverse to $\xi$ because $h$ is nowhere-zero on $M$. As a result, there must be a nowhere-zero function $f:M\to\R$ such that $$\alpha(R)=f.$$ To check this, assume, on the contrary, that the function $M \to \R$ given by $x \mapsto \alpha_x(R_x)$ has a zero, say at $p$. So, we have $\alpha_p(R_p)=0$ which means that $R_p \in \textrm{Ker}(\alpha_p)=\xi_p\,$. But this contradicts to the fact that $R$ is everywhere transverse to $\xi$. Hence, we obtain a contact$-G_2-$structure $(\varphi,R,\alpha,f,1)$. This finishes the proof.

\end{proof}

%
%

\section{Some examples} \label{sec:Examples}
In this final section, we give some examples of $G_{2}$-manifolds
admitting $A$-compatible contact structures. In fact, by Theorem \ref{thm:A_compatible-->Contact_G2_Str}, in each example we will also have a corresponding contact$-G_2-$structure.

\vspace{.1in}

\subsection{$CY \times S^{1}$ (or $CY \times \R$)} Consider a well-known example of $G_{2}$-manifold
$(CY\times S^{1}, \varphi)$ where we assume $CY(\Omega, \omega)$ is a 3-fold Calabi-Yau manifold which is either noncompact or compact with boundary. Assume K\"ahler form $\omega$ on $CY$ is exact, i.e. $\omega=d\lambda$ for some $\lambda \in \Omega^{1}( CY)$ and set $\alpha=dt + \lambda$ where $t$ is the
coordinate on $S^1$. Then $\alpha \wedge (d\alpha)^3=\omega^3 \wedge
dt$ is a volume form, and so $\alpha$ is a contact 1-form on
$CY\times S^{1}$. Moreover, $\partial t$ is the Reeb vector field of
$\alpha$ as $\iota_{\partial t} \alpha=1$ and $\iota_{\partial t}
d\alpha=\iota_{\partial t} \omega=0$. Also observe that since
$\varphi=Re(\Omega) + \omega \wedge dt$ (see \cite{Joyce}, for
instance), we compute

\begin{center}
$\iota_{\partial t} \varphi=\iota_{\partial t}(Re(\Omega)+ \omega
\wedge dt)=\iota_{\partial t}Re(\Omega)+ \iota_{\partial t}(\omega
\wedge dt)=\omega \iota_{\partial t}dt=\omega=d\lambda=d\alpha$.
\end{center}

Thus, $\xi=$ Ker$(\alpha)$ is an A-compatible contact structure on
($CY\times S^{1}, \varphi)$, or in other words, $(\varphi,\partial t, \alpha, 1,1)$ is a contact$-G_2-$structure on $CY\times S^{1}$. We note that, by considering $t$ as a
coordinate on $\R$, the above argument also gives a contact$-G_2-$structure on $CY \times \R$.

\vspace{.1in}

\subsection{$W \times S^{1}$ (or $W \times \R$)} We now give a special case of the above example. First, we need some definitions: A \emph{Stein manifold} of complex dimension $n$ is a triple $(W^{2n},J,\psi)$ where $J$ is a complex structure on $W$ and $\psi: W
\rightarrow \R$ is a smooth map such that the $2-$form
$\omega_\psi=-d(d\psi \circ J)$ is non-degenerate (and so an exact
symplectic form) on $W$. Indeed, $(W,J,\omega_{\psi})$ is an exact
K\"ahler manifold. We say that $(M^{2n-1},\xi)$ is \emph{Stein
fillable} if there is a Stein manifold $(W^{2n},J,\psi)$ such that
$\psi$ is bounded from below, $M$ is a non-critical level of $\psi$,
and $-(d\psi\circ J)$ is a contact form for $\xi$.

\vspace{.1in}

Next, consider a parallelizable Stein manifold $(W,J,\psi)$ of
complex dimension three. By a result of \cite{Forster}, we know that
$c_1(W,J)=0$, i.e., the first Chern class of $(W,J)$ vanishes.
Therefore, $W$ admits a Calabi-Yau structure with associated
K\"ahler form $\omega_\psi=-d(d\psi \circ J)$. Let $\Omega$ be the
non-vanishing holomorphic $3$-form on $W$ corresponding to this
Calabi-Yau structure. Then by the previous example, $(W\times S^{1},
\varphi)$ is a $G_2$-manifold with $\varphi=Re(\Omega) +
\omega_{\psi} \wedge d\theta$ (where $\theta$ is the coordinate on
$S^1$), $\alpha=d\theta-(d\psi \circ J)$ is a contact 1-form on
$W\times S^{1}$ with the Reeb vector field $\partial \theta$, and
$\xi=$ Ker$(\alpha)$ is an A-compatible contact structure on
($W\times S^{1}, \varphi)$. Again by considering $\theta$ as a
coordinate on $\R$, we obtain an A-compatible contact structure on
$(W \times \R, \varphi)$. Note that the corresponding contact$-G_2-$structure in both cases is $(\varphi,\partial \theta, \alpha, 1,1)$.

Now consider the unit disk $\D^2 \subset \C$. Then $(W \times \D^2,J
\times i,\psi+|z|^2)$ is a Stein manifold where $i$ is the usual
complex structure and $z=re^{i\theta}$ is the coordinate on $\C$.
Let $\eta$ be the induced contact structure on the boundary
$$\partial (W \times \D^2)=(\partial W \times \D^2) \cup (W \times
S^1).$$ Then we remark that the restriction of the Stein fillable structure $\eta$ on $W \times
S^1$ is the contact structure $\xi$ constructed above.

\vspace{.12in}

\subsection{$\mathbb{R}^{3} \times K^{4}$} Let $K$ be a K\"ahler manifold with an exact K\"ahler form $\omega$, i.e. $\omega=d\lambda$ for some
$\lambda \in \Omega^{1}(K)$. Note that $K$ is either noncompact or compact with boundary.
Consider the $G_{2}$-manifold
$\mathbb{R}^{3} \times K^{4}$ with the $G_{2}$-structure
$$\varphi=dx_{1}dx_{2}dx_{3}+\omega \wedge dx_{1}+Re(\Omega) \wedge
dx_{2}-Im(\Omega) \wedge dx_{3}$$ where $(x_1,x_2,x_3)$ are the
coordinates on $\R^3$ (see \cite{Joyce}). Then $\alpha=dx_{1}+x_{2}dx_{3}+\lambda$ is a
contact 1-form as $\alpha \wedge (d\alpha)^3=dx_{1} dx_{2} dx_{3}
\wedge \omega^2$ is a volume form on $\mathbb{R}^{3} \times K^{4}$.
One can easily check that $\partial x_{1}$ is the Reeb vector field
of $\alpha$. Furthermore,
\begin{eqnarray}
\nonumber i_{\partial x_{1}} \varphi &=& i_{\partial
x_{1}}(dx_{1}dx_{2}dx_{3}+\omega \wedge dx_{1}+\omega \wedge
dx_{2}+\omega \wedge
dx_{3})\\
\nonumber &=&dx_{2}dx_{3}+i_{\partial x_{1}}(\omega
dx_{1})=dx_{2}dx_{3}+ \omega = d(x_{2}dx_{3}+\lambda)=d\alpha
\end{eqnarray}
Hence, $\xi=$ Ker$(\alpha)$ is an A-compatible contact structure on
($\mathbb{R}^{3} \times K^{4}, \varphi)$ and the corresponding contact$-G_2-$structure on $\mathbb{R}^{3} \times K^{4}$ is $(\varphi,\partial x_1, \alpha, 1,1)$.

\vspace{.12in}

\subsection{$T^*M^3 \times \R$} Let $M$ be any oriented Riemannian $3$-manifold and $T^*M$ denote the cotangent bundle of $M$. It is shown in \cite{Cho-Salur-Todd} that $T^*M \times \R$ has a $G_2$-structure $\varphi$ with $d\varphi=0$. To describe $\varphi$, let $(x_1,x_2,x_3)$ be local coordinates on $M$ around a given point, and consider the corresponding standard local coordinates $(x_1,x_2,x_3,\xi_1,\xi_2,\xi_3)$ on the cotangent bundle $T^*M$. These define the standard symplectic structure $\omega=-d \lambda$ on $T^*M$ where $\lambda=\Sigma_{i=1}^3 \xi_i dx_i$ is the tautological 1-form on $T^*M$. Let $t$ denote the coordinate on $\R$. Then $\varphi=Re(\Omega) - \omega \wedge dt$ where $\Omega=(dx_1 + id\xi_1)\wedge (dx_2 + id\xi_2)\wedge (dx_3 + id\xi_3)$ is the complex-valued $(3,0)$-form on $M$. On the other hand, the $1$-form $\alpha=dt+\lambda$ is a contact form on $T^*M \times \R$ with the Reeb vector field $\partial t$. Now it is straightforward to check that $\xi=\textrm{Ker}(\alpha)$ is an $A$-compatible contact structure on $(T^*M \times \R,\varphi)$ and also that $(\varphi,\partial t, \alpha, 1,1)$ is the corresponding contact$-G_2-$structure on $T^*M \times \R$.


\end{document}